\documentclass{amsart}

\usepackage{amsmath}
\usepackage{textcomp}
\usepackage{amsfonts}
\usepackage{amssymb,enumerate}
\usepackage{amsthm}
\usepackage{stmaryrd}
\usepackage[all]{xy}

\newcommand{\Ann}{\operatorname{Ann}}
\newcommand{\Ass}{\operatorname{Ass}}
\newcommand{\Jac}{\operatorname{Jac}}

\newcommand{\Spec}{\operatorname{Spec}}
\newcommand{\Supp}{\operatorname{Supp}}

\newcommand{\depth}{\operatorname{depth}}

\newcommand{\rr}{R\bowtie^f J}
\newcommand{\fd}{\operatorname{f-depth}}
\newcommand{\gd}{\operatorname{g-depth}}

\newcommand{\fm}{\frak{m}}
\newcommand{\fp}{\frak{p}}

\newcommand{\fq}{\frak{q}}
\newcommand{\fa}{\frak{a}}

\newtheorem{thm}{Theorem}[section]
\newtheorem{cor}[thm]{Corollary}
\newtheorem{lem}[thm]{Lemma}
\newtheorem{prop}[thm]{Proposition}

\begin{document}
	
	\bibliographystyle{amsplain}

	\date{}
	
	\author{Y. Azimi}

	\address{Department of Mathematics, University of Tabriz,
		Tabriz, Iran.} \email{u.azimi@tabrizu.ac.ir}
	
	


	\title[(generalized)  filter properties on amalgamated algebra]{(generalized)  filter properties of the amalgamated algebra}

	\begin{abstract}
		Let $R$ and $S$ be commutative rings with unity, $f:R\to S$ a ring homomorphism and $J$ an ideal of $S$. Then the subring $R\bowtie^fJ:=\{(a,f(a)+j)\mid a\in R$ and $j\in J\}$ of $R\times S$ is called the amalgamation of $R$ with $S$ along $J$ with respect to $f$. In this paper, we determine when  $R\bowtie^fJ$ is a (generalized) filter ring.
	\end{abstract}
	
\maketitle
	
\section{Introduction}
Throughout  this paper, 
let $R$ and $S$ be two
commutative rings with identity,
 $J$ be a non-zero
proper ideal of $S$, and 
$f:R\to S$ be a ring homomorphism.

D'Anna, Finocchiaro, and Fontana in
 \cite{DFF} and \cite{DFF2} have
introduced  the following
subring (with standard component-wise operations)
$$R\bowtie^fJ:=\{(r,f(r)+j)\mid r\in R\text{ and }j\in J\}$$ 
of
$R\times S$, called the \emph{amalgamated algebra} (or \emph{amalgamation}) of $R$ with $S$ along $J$
with respect to $f$. This construction generalizes the amalgamated
duplication of a ring along an ideal
 (introduced and studied in
\cite{DF}). Moreover, several classical
 constructions such as Nagata's
idealization (cf. \cite[page 2]{Na}), 
the $R + XS[X]$ and
the $R+XS\llbracket X \rrbracket$ constructions
 can be studied as
particular cases of this construction
 (see \cite[Example 2.5 and Remark 2.8]{DFF}).
Recently, many properties of amalgamations investigated in several papers
(e.g. \cite{sss16}, \cite{ass17}, \cite{ass}, \cite{acat}, etc.)
and
the construction has proved its worth 
providing numerous (counter)examples
in commutative ring theory.

In \cite{cst}, Cuong et al. introduced
 the notion of filter regular sequence as an
extension of regular sequence, and via this notion,
 they studied \emph{$f$-modules}, as an
 extension of (generalized) Cohen-Macaulay modules.
 This structure is a well-known structure in commutative algebra and
 have applications in algebraic geometry.
 Then, in \cite{nh}, Nhan extended this notion
 to generalized regular sequence, which
 in turn,  leads  to the introduction of
  \emph{generalized $f$-modules} in \cite{nhm}.
We have the following implications:
\begin{center}
	Gorenstein ring $\Longrightarrow$
	Cohen-Macaulay ring $\Longrightarrow$
	generalized Cohen-Macaulay ring $\Longrightarrow$
	$f$-ring $\Longrightarrow$
	generalized $f$-ring.
\end{center}
It has already investigated that when
$\rr$ is one of the  three first 
in the above list (\cite{ass17}, \cite{ass19},\cite{ass}, \cite{A}).
In this paper, we investigate when it
is one of the two last properties.

The proofs for the two  case is almost
the same, but for $f$-modules easier.
Therefore we deal with case of
generalized $f$-modules in details, 
and the same  proof with minor modifications works in the case
of $f$-modules.  We provide a sketch of
proof for this case and leave details for the reader.

\section{Results}

Let us first fix some notation which we shall use throughout the
 paper:
 As mentioned above, $R$ and $S$ are two commutative  rings with identity, 
  $J$ is an  ideal of the ring $S$, and $f:R\to S$ is a ring homomorphism. 
In the sequel, we consider contractions
and extensions with respect to
the natural embedding
$\iota _R: R\to R\bowtie^f J$ defined by
$\iota _R (x)=(x,f(x))$, for every
$x\in R$.

Let $I$ be an ideal of $R$, and $M$ be a finitely
generated  $R$-module such that $M \neq IM$.
 We shall refer to the length of a maximal $M$-sequence contained
in $I$ as the depth of $M$ in $I$, and we shall denote this by $\depth (I,M)$.
 It will be
 convenient to use $\depth M$ to denote
 $\depth (\fm,M)$ when $(R,\fm)$
 is a  local ring.

(Generalized) $f$-modules are defined in the 
context of Noetherian local rings for
finitely generated modules. Thus we always
assume that $(R,\fm)$ is a Noetherian local
  ring and $J$ is 
 finitely generated as an $R$-module.
We will also assume that $J\subseteq \Jac (S)$.
When this is the case,
$(\rr , \fm^{\prime_f})$ is also a Noetherian local ring (see \cite[Proposition 5.7]{DFF} and
\cite[Corollary 2.7]{DFF1}).

The notion of $M$-\emph{generalized regular sequence} of $M$ is defined as a sequence  $x_1,\ldots, x_n$ of elements in $\fm$ such that,  for all $i = 1,\ldots, n$, $x_i \notin \fp$ for all $\fp \in \Ass (M/(x_1,\ldots, x_{i-1})M)$ satisfying $\dim R/\fp >1$.
The length of a maximal generalized regular sequence of $M$ in $I$ is
called the \emph{generalized depth of $M$ in $I$} and denoted by $\gd (I,M)$.
In this paper, we use the following characterization for $\gd (I,M)$ by 
the support of local cohomology module $H^i_{I}(M)$:

\begin{lem}\label{0}
	Let $I$ be
	an ideal of $R$, and $M$ be a finitely generated $R$-module. 
	Then the following equality holds.
	$$\gd(I, M)=\min \{r\mid \text{there exists } \fp \in \Supp_R (H^r_{I}(M))  \text{ such that }\dim R/\fp >1\}.$$
\end{lem}
\begin{proof}
	If $\dim M/I M>1$, then the assertion holds by \cite[Proposition 4.5]{nh}.
	If $\dim M/I M\leq 1$, then
	by definition, $\gd(I, M)=\infty$.
	The other side is also infinite since  
	$\Supp_R (H^r_{I}(M)) \subseteq 
	\Supp (M) \cap \Supp (R/I)=\Supp (M/IM)$.
\end{proof}

The following lemma, which has the key role
in the proof of Theorem \ref{gd},
 links the $\gd$ of
$\rr$ in the extension ideal $\fa^e$
to the  $\gd$ of $R$ and
$J$ in the prime ideal $\fa$.

\begin{lem}\label{gdlem}
	Let $\fa \in \Spec (R)$.  Then 
	the following holds:
	$$\gd(\fa^e,\rr)=\min \{\gd (\fa,R),\gd (\fa,J)\}.$$	
\end{lem}
\begin{proof}
	We first show that 
	the existence of some  $\mathcal{P} \in \Supp_{\rr} \left(H^r_{\fa^e}(R\bowtie^f J)\right)$ with the property $\dim \rr/\mathcal{P} >1$	
	is equivalent to the existence of some
	$\fp \in \Supp_R \left(H^r_{\fa^e}(\rr)\right)$ 
	with the property $\dim R/\fp >1$.
	To achieve this, first we
	note that, by \cite[Lemma 3.6]{DFF2},
	the extension $\iota _R: R\to R\bowtie^f J$ is  integral  
	since we assume that $J$ is
	finitely generated as an $R$-module. 
	Therefore, for any $\mathcal{P} \in \Spec (\rr)$,
	we have
	 $\dim \rr/\mathcal{P} >1$ if and only if
	$\dim R/\mathcal{P}^c >1$.
	Next,	
	let $\mathcal{P} \in \Supp_{\rr} \left(H^r_{\fa^e}(\rr)\right)$, 
	say $\alpha/1$ is a non-zero element
	of $\left(H^r_{\fa^e}(\rr)\right)_\mathcal{P}$.
	If $r\in R$ such that $r\alpha =0$, then
	$f(r)\in \mathcal{P}$, i.e. $r\in \mathcal{P}^c$.
	We have thus proved $\mathcal{P}^c \in \Supp_R \left(H^r_{\fa^e}(\rr)\right)$.\\
	Suppose conversely that 
	$\fp \in \Supp_R \left(H^r_{\fa^e}(\rr)\right)$.
	Then, for some
	ideal $\mathcal{I}$ of $\rr$, with the property  
	$\rr / \mathcal{I}\subseteq H^r_{\fa^e}(\rr)$,
	we have
	$\fp \in \Supp_R \left(\rr / \mathcal{I}\right)$. 
	From this we have
	$\mathcal{I}^c\subseteq \fp$.
	By lying over property, there exists $\mathcal{P} \in \Spec (\rr)$
	such that $\mathcal{I}\subseteq \mathcal{P}$ and $\mathcal{P}^c=\fp$,
	hence that $\mathcal{P} \in \Supp_{\rr} \left(\rr / \mathcal{I}\right) \subseteq \Supp_{\rr} \left(H^r_{\fa^e}(\rr)\right)$.
	 This completes the proof of our claim. Now we have:	
	\begin{align*}
	\gd(\fa^e, \rr )
	&=\min \{r| \exists  \mathcal{P} \in \Supp_{\rr} \left(H^r_{\fa^e}(\rr)\right);\  \dim \rr/\mathcal{P} >1\}\\
	&=\min \{r| \exists  \fp \in \Supp_{R} \left(H^r_{\fa^e}(\rr)\right);\ \dim R/\fp >1\}\\
	&=\min \{r| \exists  \fp \in \Supp_{R} \left(H^r_{\fa}(\rr)\right);\ \dim R/\fp >1\}\\
	&=\min\{r|\exists  \fp \in \Supp_{R} \left(H^r_\fa (R)\oplus H^r_\fa (J)\right);\ \dim R/\fp >1\}\\
	&= \min \{\gd (\fa,R),\gd (\fa,J)\}.
	\end{align*}
	The first and last equality hold
	by Lemma \ref{0}, while	
	the second one holds by the above
	observation.  The third equality follows by the Independence Theorem of local cohomology
	\cite[Theorem 4.2.1]{BS}, and
	the forth
	equality obtained using the $R$-module
 	isomorphism $\rr \cong R\oplus J$ \cite[Lemma 2.3]{DFF}.
	
\end{proof}
\emph{Generalized $f$-modules} were introduced in \cite{nhm} as modules for which every system of parameters is
a generalized regular sequence. A ring is called a 
\emph{generalized $f$-ring} if it is a generalized $f$-module over itself.
For more details we refer the reader to
\cite{nh} and \cite{nhm}.
We define  a finitely generated $R$-module $M$ to be 
\emph{maximal generalized $f$-module} if
$\gd (\fp,M)=\dim (R) - \dim (R/\fp)$, 
for any $\fp \in \Supp M$ satisfying
$\dim R/\fp >1$.
This definition has stem in the following 
proposition.
\begin{prop}\label{nhm1}
	Assume that $M$ is a finitely generated $R$-module
	 such that $\dim M>1$. Then the following statements are equivalent:
	\begin{itemize}
		\item [(1)]
		$M$ is a generalized $f$-module.
		\item [(2)]
		$\gd (\fp,M)=\dim (M) - \dim (R/\fp)$
		for each $\fp \in \Supp M$ satisfying
		$\dim R/\fp >1$.
		\item [(3)]
		$\gd (I,M)=\dim (M) - \dim (R/I)$
		for any proper ideal $I$ of $R$
		satisfying $I\supseteq \Ann (M)$ and $\dim R/I >1$.
	\end{itemize}
\end{prop}
\begin{proof}
	$(1) \Rightarrow (2)$ and $(3)\Rightarrow(1)$
	is by \cite[Proposition 2.5]{nhm}. The proof 
	of $(2) \Rightarrow (3)$ is similar to 
	the proof of \cite[Remark 4.2]{lt},
	using \cite[Proposition 4.3 (ii)]{nh} 
	and \cite[Proposition 2.5]{nhm}.
\end{proof}

We use the above proposition to  investigate
when $\rr$ is a generalized $f$-ring, which is 
 one of  our main results.
Recall that a finitely generated module $M$ over a Noetherian local
ring $(R,\fm)$ is called a \emph{maximal Cohen-Macaulay $R$-module} if
$\depth M=\dim R$.
In the sequel, when we consider $J$ as a module,
we always consider it as an $R$-module via
the homomorphism  $f:R\to S$. 
In particular, by $\Supp J$ we mean
 $\Supp_R J$.
\begin{thm}\label{gd}
The  following statements are equivalent:
\begin{itemize}
	\item [(1)]
	$\rr$ is a generalized $f$-ring.
	\item [(2)]
	$R$ is a generalized $f$-ring and  $J$ is a maximal generalized $f$-module.
	\item [(3)]
	$R$ is a generalized $f$-ring and
	$J_\fp$ is  maximal Cohen-Macaulay for
	 any $\fp \in \Supp (J)$ satisfying
	 $\dim R/\fp >1$.
\end{itemize}
\end{thm}
\begin{proof}
	We first assume that $\dim J>1$. The 
	process of proof shows that the opposite
	assumption,  $\dim J\leq 1$, leads to  
	trivial cases.\\
	$(1) \Rightarrow (2)$
	Assume that $\rr$ is a generalized $f$-ring and 
	pick $\fp \in \Spec (R)$ satisfying
	$\dim R/\fp >1$.  
	By \cite[Lemma 3.6]{DFF2},
	$\iota _R: R\to R\bowtie^f J$ is 
	an integral extension.
	Hence, by lying over property, 	$\fp= \fp^{ec}$, 
	hence that 
		$\dim \rr/\fp^e = \dim R/\fp>1$.
	Now, by Proposition \ref{nhm1} and Lemma \ref{gdlem}, we have:
	\begin{align*}
	\dim R - \dim R/\fp&= \dim \rr - \dim \rr/\fp^e\\
	&=\gd (\fp^e,\rr)\\
	&\leq \gd (\fp,R)\\
	&\leq \dim R - \dim R/\fp.
	\end{align*}	
	Again we use Proposition \ref{nhm1} to see that	
	$R$ is a generalized $f$-ring, and a similar argument will show that $J$ is a maximal generalized $f$-module.
	
	$(2) \Rightarrow (1)$ Suppose that
	$R$ is a generalized $f$-ring and  $J$ is a maximal generalized $f$-module.
    Then, from Lemma \ref{gdlem} and Proposition \ref{nhm1},   we deduce that 
	$\gd (\fp^e,\rr)=\gd (\fp,R)$, for any $\fp\in\Spec (R)$.	
	Now, let $\mathcal{P}\in \Spec (\rr)$ and 
	$\dim \rr/\mathcal{P}>1$.
	Then $\dim R/\mathcal{P}^c>1$
	and, by Lemma \ref{gdlem} and Proposition \ref{nhm1}, we have:
	\begin{align*}
	\dim \rr - \dim \rr/\mathcal{P}&=
	\dim R - \dim R/\mathcal{P}^c\\
	&=\gd (\mathcal{P}^c,R)\\
	&= \gd (\mathcal{P}^{ce},\rr)\\
	&\le \gd (\mathcal{P},\rr)\\
	&\le \dim \rr - \dim \rr/\mathcal{P}.
	\end{align*}
	Thus  inequalities are equality, 
	and another appeal to Proposition  \ref{nhm1} gives the desired conclusion.
	
	$(2) \Rightarrow (3)$ Let $\fp \in \Supp (J)$ with the property $\dim R/\fp >1$.
	 In order to show that
	$J_\fp$ is maximal Cohen-Macaulay, observe that
	 \cite[Proposition 4.4]{nh} together
	 with our assumptions yields the following
	 inequalities:
	 $$\depth J_\fp \geq \gd (\fp,J)
	 =\dim R - \dim R/\fp \geq \dim R_\fp \geq \depth J_\fp.$$
	 $(3) \Rightarrow (2)$ Let $\fp \in \Supp (J)$
	  satisfying $\dim R/\fp >1$.
	  Then, using \cite[Proposition 4.4]{nh} and
	  \cite[Proposition 1.2.10(a)]{BH98}, we get a prime ideal $\fq$ containing $\fp$ 
	  such that $\fq \in \Supp (J)$,  $\dim R/\fq >1$, and $\gd (\fp,J)= \depth J_\fq$.
	  The following
	  inequalities complete the proof:
\begin{center}
	  $\gd (p,J)= \depth J_\fq = \dim R_\fq \geq
	  \gd (\fq, R)=$ $\dim R - \dim R/\fq \geq
	  \dim R - \dim R/\fp \geq \gd (p,J).$
\end{center}
\end{proof}

Recall that if $f:=id_R$ is the identity homomorphism
on $R$, and $I$ is an ideal of $R$, then $R\bowtie
I:=R\bowtie^{id_R} I$ is called the
amalgamated duplication of $R$ along $I$.
The next corollary deals with this case.
\begin{cor}
	$R\bowtie I$ is a generalized $f$-ring if and only if $R$ is a generalized $f$-ring and  $I$ is maximal generalized $f$-module if and only if $R$ is a generalized $f$-ring and
	$I_\fp$ is maximal Cohen-Macaulay for
	any $\fp \in \Supp (I)$ satisfying
	$\dim R/\fp >1$.
\end{cor}

Let $M$ be an $R$-module. Nagata (1955) considered a ring
extension of $R$ called the the \emph{idealization} of $M$ in $R$
, denoted here by
$R\ltimes M$ \cite[page 2]{Na}. As in
\cite[Remark 2.8]{DFF}, if $S:=R\ltimes M$, $J:=0\ltimes M$, and
$\iota:R\to S$ be the natural embedding, then $R\bowtie^\iota J\cong
R\ltimes M$.
It is easy to check that, as $R$-modules,
$0\ltimes M \cong M$.
The following corollary shows
when the idealization is generalized $f$-ring.

\begin{cor}
	If $M$ is a finitely generated $R$-module, then
$R\ltimes M$ is a generalized $f$-ring  if and only if $R$ is a generalized $f$-ring and  $M$ is a maximal generalized $f$-module	if and only if
$R$ is a generalized $f$-ring and
$M_\fp$ is maximal Cohen-Macaulay for
any $\fp \in \Supp M$ satisfying
$\dim R/\fp >1$.
\end{cor}

In the remaining part of the paper we investigate
when $\rr$ is an $f$-ring. The arguments are
the same as the ones in the case of generalized $f$-ring.
But, for the reader’s convenience, we give brief proofs and refer the reader
to previous arguments.

The notion of $M$-\emph{filter regular sequence} is defined as a sequence  $x_1,\ldots, x_n$ of elements in $\fm$ such that $x_i \notin \fp$ for all $\fp \in \Ass (M/(x_1,\ldots, x_{i-1})M)\setminus \{\fm\}$ and for all $i = 1,\ldots, n$.
The \emph{filter depth}, $\fd (I,M)$, of $I$ on $M$ is defined as the length of any maximal $M$-filter regular sequence in $I$. Here, we use the following characterization for $\fd (I,M)$ (see \cite[Theorem 3.1]{M} and \cite[Theorem 3.10]{lt}):
$$\fd(I, M )=\inf\{r\mid H^r_{I}(M) \text{ is not an Artinian} \ R\text{-module}\}.$$

The following lemma expresses $\fd(\fp^e,\rr)$, the $\fd$ of extension of a prime ideal $\fp$ of $R$ in $\rr$. For the proof, we use the elementary fact that being Artinian as an $\rr$-module
is the same as being Artinian as an $R$-module. 
\begin{lem}\label{fdlem}
Let $\fp\in \Spec (R)$. Then
the following holds:
  $$\fd(\fp^e,\rr)=\min \{\fd (\fp,R),\fd (\fp,J)\}.$$	
\end{lem}
\begin{proof}
By
 \cite[Theorem 3.10]{lt} (and arguments similar to Lemma \ref{gdlem}), we have:
\begin{align*}
\fd(\fp^e, \rr )
&=\inf\{r|H^r_{\fp^e}(\rr) \text{ is not Artinian} \ \rr\text{-module}\}\\
&=\inf\{r|H^r_{\fp^e}(\rr) \text{ is not Artinian} \ R\text{-module}\}\\
&=\inf\{r|H^r_\fp(\rr) \text{ is not Artinian} \ R\text{-module}\}\\
&=\inf\{r|H^r_\fp(R)\oplus H^r_\fp(J)  \text{ is not Artinian}\ R\text{-module}\}\\
&= \min \{\fd (\fp,R),\fd (\fp,J)\}.
\end{align*}
\end{proof}
In \cite{cst}, the authors introduced
\emph{f-modules} as modules for which every system of parameters is a filter regular sequence. The ring $R$ is called an \emph{$f$-ring} if it is an $f$-module over itself.
This structure is a well-known structure in commutative algebra and have applications in algebraic geometry. For more details we refer the reader to \cite{cst}, \cite{SV}, and \cite{lt}.
We define an $R$-module $M$ to be 
\emph{maximal $f$-module} if
 $\fd (\fp,M)=\dim (R) - \dim (R/\fp)$, 
 for any $\fp \in \Supp M\setminus \{\fm\}$.
This definition has stem in the following 
proposition \cite[Theorem 4.1 and Remark 4.2]{lt}:
\begin{prop}\label{lutang}
	For a finitely generated $R$-module $M$,  the following statements are equivalent:
	\begin{itemize}
		\item [(1)]
		$M$ is an $f$-module
		\item [(2)]
		for any $\fp \in \Supp M\setminus \{\fm\}$,
		$\fd (\fp,M)=\dim (M) - \dim (R/\fp)$
		\item [(3)]
		for any proper ideal $I$ of $R$
		with the property $I\supseteq \Ann (M)$ and $\sqrt{I}\neq \fm$,
		$\fd (I,M)=\dim (M) - \dim (R/I)$
	\end{itemize}
\end{prop}

We use the above proposition to  investigate
when $\rr$ is $f$-ring, which is our final result.
 
\begin{thm}\label{fd}
The following statements are equivalent:
	 \begin{itemize}
	 	\item [(1)]
	 	$\rr$ is an $f$-ring.
	 	\item [(2)]
	 	$R$ is an $f$-ring and  $J$ is a maximal  $f$-module.
	 	\item [(3)]
	 	$R$ is an $f$-ring and
	 	$J_\fp$ is  maximal Cohen-Macaulay for any $\fp \in \Supp (J)\setminus \{\fm\}$.
	 \end{itemize}
\end{thm}
\begin{proof}
	$(1) \Rightarrow (2)$
	Assume that $\rr$ is an $f$-ring and 
pick $\fp \in \Spec (R) \setminus \{\fm\}$.  	
As before, the extension
$\iota _R: R\to R\bowtie^f J$ is 
 integral, and so 
$\fp= \fp^{ec}$. 
Thus  $\sqrt{\fp^e}\neq \fm^{\prime_f}$
and $\dim \rr/\fp^e = \dim R/\fp$.
Then Proposition \ref{lutang} gives the desired conclusion, just as in the proof of Theorem \ref{gd}.

$(2) \Rightarrow (1)$ Suppose	that
$R$ is an $f$-ring and  $J$ is a maximal $f$-module, and
let $\mathcal{P}\in \Spec (\rr)\setminus \{\fm^{\prime_f}\}$.
Then $\mathcal{P}^c\in \Spec (R) \setminus \{\fm\}$
and  Proposition \ref{lutang} gives the desired
conclusion, as in the case of Theorem \ref{gd}.

$(2)\Leftrightarrow (3)$ The proof of this part
is the same as the proof in Theorem \ref{gd}, 
using the following equality instead of
\cite[Proposition 4.4]{nh}:
$$\fd (\fp , J)= \min \{\depth (\fp R_\fq ,J_\fq) \mid \fq \in \Supp (J/\fp J)\setminus \{ \fm\} \}.$$
For the proof the 
equality, see the proof of \cite[Theorem 3.10]{lt}.
\end{proof}

\begin{cor}
	(cf. \cite[Theorem 3.5]{SSh}.)
	$R\bowtie I$ is an  $f$-ring  if and only if $R$ is an  $f$-ring and  $I$ is maximal  $f$-module  if and only if $R$ is an $f$-ring and
	$I_\fp$ is maximal Cohen-Macaulay for any $\fp \in \Supp (I)\setminus \{\fm\}$.
\end{cor}

\begin{cor}
	If $M$ is a finitely generated $R$-module, then
	$R\ltimes M$ is an $f$-ring  if and only if $R$ is an $f$-ring and  $M$ is a maximal  $f$-module	if and only if
	$R$ is an $f$-ring and
	$M_\fp$ is maximal Cohen-Macaulay for
	any $\fp \in \Supp (M)\setminus \{\fm\}$.
\end{cor}

\end{document}